\documentclass[oneside,english]{amsart}
\usepackage[T1]{fontenc}
\usepackage[latin9]{inputenc}
\usepackage{units}
\usepackage{mathrsfs}
\usepackage{amsthm}
\usepackage{amstext}
\usepackage{amssymb}
\usepackage{esint}

\makeatletter
\numberwithin{equation}{section}
\numberwithin{figure}{section}
\usepackage{enumitem}		
\usepackage{color}

\def\@begintheorem#1#2[#3]{%
\deferred@thm@head{\the\thm@headfont \thm@indent
\@ifempty{#1}{\let\thmname\@gobble}{\let\thmname\@iden}%
\@ifempty{#2}{\let\thmnumber\@gobble}{\let\thmnumber\@iden}%
\@ifempty{#3}{\let\thmnote\@gobble}{\let\thmnote\@iden}%
\thm@swap\swappedhead\thmhead{#1}{#2}{#3}%
\the\thm@headpunct
\thmheadnl 
\hskip\thm@headsep
}%
\ignorespaces
{
\@ifempty{#3}{\addcontentsline{toc}{subsection}{#1 #2}}
{\addcontentsline{toc}{subsection}{#1 #2 (#3)}}
}
}
\definecolor{thmcolor}{rgb}{0,0.00,0.5}

\newtheoremstyle{thmsty}%
{3pt}
{6pt}
{}
{}
{\color{thmcolor}\sc\bfseries}
{.}
{.5em}
{}

\theoremstyle{thmsty}

\newtheorem{thm}[subsection]{Theorem}
\definecolor{defcolor}{rgb}{0.75,0.25,0}

\newtheoremstyle{defsty}%
{3pt}
{6pt}
{}
{}
{\color{defcolor}\sc\bfseries}
{.}
{.5em}
{}

\theoremstyle{defsty}

\newtheorem{defn}[subsection]{Definition}
\definecolor{propcolor}{rgb}{0,0.25,0.75}

\newtheoremstyle{propsty}%
{3pt}
{6pt}
{}
{}
{\color{propcolor}\sc\bfseries}
{.}
{.5em}
{}

\theoremstyle{propsty}

\newtheorem{prop}[subsection]{Proposition}
\usepackage{bbding}
\renewenvironment{proof}[1][\proofname]{\par
\pushQED{\qed}
\normalfont \topsep6\p@\@plus6\p@\relax
\trivlist
\item[\hskip\labelsep
\textbf{\emph
{#1\@addpunct{.}}}]\ignorespaces
}{\popQED\endtrivlist\@endpefalse}

\definecolor{corcolor}{rgb}{0,0.5,0.75}

\newtheoremstyle{corsty}%
{3pt}
{6pt}
{}
{}
{\color{corcolor}\sc\bfseries}
{.}
{.5em}
{}

\theoremstyle{corsty}

\newtheorem{cor}[subsection]{Corollary}

\makeatother

\usepackage{babel}

\begin{document}
\global\long\def\v#1{\vec{\mathbf{#1}}}
\global\long\def\norm#1{\left\Vert #1\right\Vert }
\global\long\def\P{\mathcal{Q}}
\global\long\def\Sch{\mathscr{S}}

\title[Wavefront sets of convolutions with (weighted) line integrals]{Wavefront sets of convolutions of distributions with (weighted) line
integral distributions}

\author{Brian Sherson}

\maketitle
Excluding pathological cases of curves $\v{\gamma}\in\mathcal{C}^{\infty}\left(\left(-\varepsilon,\infty\right),\mathbb{R}^{n}\right)$,
for some $\varepsilon>0$, the following line integral:
\[
\left\langle \mathcal{T}_{\v{\gamma}},\phi\right\rangle =\int_{0}^{\infty}\phi\left(\v{\gamma}\left(t\right)\right)\norm{\v{\gamma}'\left(t\right)}\, dt,\quad\phi\in\Sch\left(\mathbb{R}^{n}\right),
\]
defines a distribution. Moreover, if we replace $\norm{\v{\gamma}'\left(t\right)}$
with any bounded positive weight function $\upsilon\in\mathcal{C}^{\infty}\left(\left(-\varepsilon,\infty\right)\right)$
for some $\varepsilon>0$, the following also defines a distribution:
\[
\left\langle \mathcal{T}_{\v{\gamma},\upsilon},\phi\right\rangle =\int_{0}^{\infty}\phi\left(\v{\gamma}\left(t\right)\right)\upsilon\left(t\right)\, dt,\quad\phi\in\Sch\left(\mathbb{R}^{n}\right),
\]
provided that either $\norm{\v{\gamma}'\left(t\right)}$ is bounded
away from zero, or $\upsilon$ decays sufficiently fast so that the
above integral converges for any $\phi\in\Sch\left(\mathbb{R}^{n}\right)$.

The convolution of distributions $u_{1}$ and $u_{2}$, once of which
has compact support, is defined in Hörmander\cite{hormander1990AnaLinParDifOpeIDisTheFouAna},
as the unique distrubtion $u=u_{1}\star u_{2}$ satisfying:
\[
u_{1}\star\left(u_{2}\star\phi\right)=u\star\phi,\quad\phi\in\mathcal{C}_{0}^{\infty}\left(\mathbb{R}^{n}\right).
\]

As such, $w\star\mathcal{T}_{\gamma}$ and $w\star\mathcal{T}_{\gamma,\upsilon}$
are well-defined as distributions whenever $w\in\mathcal{E}'\left(\mathbb{R}^{n}\right)$.
This will give rise to formally defining the notations 
\begin{flalign*}
 & \hfill & w\star\mathcal{T}_{\v{\gamma}}\left(\v x\right) & =\int_{0}^{\infty}w\left(\v x-\v{\gamma}\left(t\right)\right)\norm{\gamma'\left(t\right)}\, dt, & \hfill\\
 &  & w\star\mathcal{T}_{\v{\gamma},\upsilon}\left(\v x\right) & =\int_{0}^{\infty}w\left(\v x-\v{\gamma}\left(t\right)\right)\upsilon\left(t\right)\, dt,
\end{flalign*}
both of which will agree with the usual notion of an integral converging
for almost every $\v x\in\mathbb{R}^{n}$ whenever $w\in\mathcal{L}^{1}\left(\mathbb{R}^{n}\right)$.
We will explore such convolutions, as well as their wavefront sets,
particularly exploring how the convolution scatters the singularities
of $w$. However, we will require a more direct formulation of such
convolutions than the definition of convolution given in Hörmander
provides for.

To avoid pathological cases, we will focus on choices of $w$ and
$\v{\gamma}$ for which given any $\v x$, $\v x-\v{\gamma}\left(t\right)$
lies outside the support of $w$ for $t$ sufficiently large.

\section{The Distributional Directional Antiderivative}

We may begin by extending the idea of directional antiderivatives
to compactly-supported distributions. In particular, given $\v v\in\mathcal{S}^{n-1}$,
we want to focus on the directional antiderivatives of the form:
\[
\mathcal{I}_{\v v}f\left(\v u+s\v v\right)=\int_{-\infty}^{s}f\left(\v u+t\v v\right)\, dt,\quad\v u\in\v v^{\perp},s\in\mathbb{R}.
\]
This is equivalent to choosing $\v{\gamma}\left(t\right)=-t\v v$.
When defining $\mathcal{I}_{\v v}w$ for a distribution $w$, we will
replace the requirement that $w$ be compactly-supported with a weaker
condition.
\begin{defn}[Distributional directional antiderivative]
\label{defn:dist-anti-partial-deriv}Let $\v v\in\mathcal{S}^{n-1}$,
$w\in\mathscr{D}'\left(\mathbb{R}^{n}\right)$, and suppose that $w$
satisfies a support condition
\begin{equation}
t_{\min}=\min_{\v x\in\mathrm{supp}\, w}\v x\cdot\v v>-\infty.\label{eq:support-condition}
\end{equation}
Now choose a $\psi_{0}\in\mathcal{C}_{0}^{\infty}\left(\mathbb{R}\right)$
with $\int_{\mathbb{R}}\psi_{0}\, dx=1$, and $\mathrm{supp}\,\psi_{0}\subseteq\left(-\infty,t_{\min}\right)$.
For $\phi\in\mathscr{S}\left(\mathbb{R}^{n}\right)$, we define
\[
\mathcal{I}_{\v v}\phi\left(\v u+t\v v\right)=\int_{-\infty}^{t}\phi\left(\v u+s\v v\right)-\mathcal{X}_{\v v}\phi\otimes\psi_{0}\left(\v u+s\v v\right)\, ds,\quad\v u\in\v v^{\perp},t\in\mathbb{R},
\]
where $\mathcal{X}_{\v v}:\mathscr{S}\left(\mathbb{R}^{n}\right)\rightarrow\mathscr{S}\left(\v v^{\perp}\right)$
denotes the x-ray transform restricted to the direction $\v v$, and
the tensor product $\mathcal{X}_{\v v}\phi\otimes\psi_{0}$ is interpreted
as:
\[
\mathcal{X}_{\v v}\phi\otimes\psi_{0}\left(\v u+s\v v\right)=\mathcal{X}_{\v v}\phi\left(\v u\right)\psi_{0}\left(s\right),\quad\v u\in\v v^{\perp},s\in\mathbb{R}.
\]

We then define the \textbf{distributional directional antiderivative}
by
\[
\left\langle \mathcal{I}_{\v v}w,\phi\right\rangle =-\left\langle w,\mathcal{I}_{\v v}\phi\right\rangle .
\]

\end{defn}
Since we defined $\mathcal{I}_{\v v}$ in a way that depends on an
arbitrary choice of $\psi_{0}$, we will want to verify that a different
choice of $\psi_{0}$ will not alter $\mathcal{I}_{\v v}$.
\begin{prop}
While $\mathcal{I}_{\v v}\phi$ depends on choice of $\psi_{0}$,
$\mathcal{I}_{\v v}w$ does not, so long as $\mathrm{supp}\,\psi_{0}\subseteq\left(-\infty,t_{\min}\right)$.\end{prop}
\begin{proof}
Let $\psi_{0},\psi_{1}\in\mathcal{C}_{0}^{\infty}\left(\mathbb{R}\right)$both
have support in $\left(-\infty,t_{\min}\right)$, and take $\mathcal{I}_{\v v}^{0}$
and $\mathcal{I}_{\v v}^{1}$ as defined above in terms of $\psi_{0}$
and $\psi_{1}$, respectively. Then for $\phi\in\Sch\left(\mathbb{R}^{n}\right)$
we observe that:
\begin{flalign*}
 & \hfill &  & \mathcal{I}_{\v v}^{0}\phi\left(\v u+t\v v\right)-\mathcal{I}_{\v v}^{1}\phi\left(\v u+t\v v\right) & \hfill\\
 &  &  & \hspace{3em}=\int_{-\infty}^{t}\phi\left(\v u+s\v v\right)-\mathcal{X}_{\v v}\phi\otimes\psi_{0}\left(\v u+s\v v\right)\, ds\\
 &  &  & \hspace{6em}-\int_{-\infty}^{t}\phi\left(\v u+s\v v\right)-\mathcal{X}_{\v v}\phi\otimes\psi_{1}\left(\v u+s\v v\right)\, ds\\
 &  &  & \hspace{3em}=\mathcal{X}_{\v v}\phi\left(\v u\right)\int_{-\infty}^{t}\left(\psi_{0}\left(s\right)-\psi_{1}\left(s\right)\right)\, ds.
\end{flalign*}
For $t$ below or above both supports of $\psi_{0}$ and $\psi_{1}$,
this integral is zero. In particular, the support of $\mathcal{I}_{\v v}^{0}\phi-\mathcal{I}_{\v v}^{1}\phi$
is contained inside $\v v^{\perp}+\left(-\infty,t_{\min}\right)\v v$.
Hence:
\begin{flalign*}
 & \hfill & \left\langle \mathcal{I}_{\v v}^{0}w,\phi\right\rangle -\left\langle \mathcal{I}_{\v v}^{1}w,\phi\right\rangle  & =-\left\langle w,\mathcal{I}_{\v v}^{0}\phi\right\rangle +\left\langle w,\mathcal{I}_{\v v}^{1}\phi\right\rangle  & \hfill\\
 &  &  & =-\left\langle w,\mathcal{I}_{\v v}^{0}\phi-\mathcal{I}_{\v v}^{1}\phi\right\rangle \\
 &  &  & =0. & \qedhere
\end{flalign*}

\end{proof}
We now wish to verify that $\mathcal{I}_{\v v}$ acts on functions
in $\mathcal{L}^{1}\left(\mathbb{R}^{n}\right)$ satisfying the support
condition \ref{eq:support-condition} in the desired manner.
\begin{prop}[Distributional Anti-partial derivative of $\mathcal{L}^{1}$ functions]
If $f\in\mathcal{L}^{1}\left(\mathbb{R}^{n}\right)$ satisfies the
support condition \ref{eq:support-condition}, then $\mathcal{I}_{\v v}f$
is in fact a function given by
\[
\mathcal{I}_{\v v}f\left(\v u+s\v v\right)=\int_{-\infty}^{s}f\left(\v u+t\v v\right)\, dt.
\]
\end{prop}
\begin{proof}
Let $t_{\min}=\min_{\v x\in\mathrm{supp\,}f}\v x\cdot\v v$, and choose
$\psi_{0}$ as described in \ref{defn:dist-anti-partial-deriv}. Observe:
\begin{flalign*}
 & \hfill & \left\langle \mathcal{I}_{\v v}f,\phi\right\rangle  & =-\left\langle f,\mathcal{I}_{\v v}\phi\right\rangle  & \hfill\\
 &  &  & =-\int_{\mathbb{R}^{n}}f\left(\v x\right)\mathcal{I}_{\v v}\phi\left(\v x\right)\, d\v x\\
 &  &  & =-\int_{\v v^{\perp}}\int_{\mathbb{R}}f\left(\v u+t\v v\right)\mathcal{I}_{\v v}\phi\left(\v u+t\v v\right)\, dt\, d\v u\\
 &  &  & =-\int_{\v v^{\perp}}\int_{\mathbb{R}}f\left(\v u+t\v v\right)\int_{-\infty}^{t}\left(\phi\left(\v u+s\v v\right)-\mathcal{X}_{\v v}\phi\otimes\psi_{1}\left(\v u+s\v v\right)\right)\, ds\, dt\, d\v u\\
 &  &  & =\int_{\v v^{\perp}}\int_{\mathbb{R}}f\left(\v u+t\v v\right)\int_{t}^{\infty}\left(\phi\left(\v u+s\v v\right)-\mathcal{X}_{\v v}\phi\otimes\psi_{1}\left(\v u+s\v v\right)\right)\, ds\, dt\, d\v u\\
 &  &  & =\int_{\v v^{\perp}}\int_{\mathbb{R}}\int_{-\infty}^{s}f\left(\v u+t\v v\right)\left(\phi\left(\v u+s\v v\right)-\mathcal{X}_{\v v}\phi\otimes\psi_{1}\left(\v u+s\v v\right)\right)\, dt\, ds\, d\v u\\
 &  &  & =\int_{\v v^{\perp}}\int_{\mathbb{R}}\int_{-\infty}^{s}f\left(\v u+t\v v\right)\, dt\,\phi\left(\v u+s\v v\right)\, ds\, d\v u\\
 &  &  & \hspace{3em}-\int_{\v v^{\perp}}\int_{\mathbb{R}}\int_{-\infty}^{s}f\left(\v u+t\v v\right)\, dt\,\psi_{0}\left(s\right)\, ds\,\mathcal{X}_{\v v}\phi\left(\v u\right)\, d\v u
\end{flalign*}
The latter integral vanishes for all $s$, since:
\[
\int_{-\infty}^{s}f\left(\v u+t\v v\right)\, dt=0,\quad s<t_{\min},
\]
and:
\[
\psi_{0}\left(s\right)=0,\quad s>t_{\min}.
\]
Hence:
\[
\left\langle \mathcal{I}_{\v v}f,\phi\right\rangle =\int_{\v v^{\perp}}\int_{\mathbb{R}}\int_{-\infty}^{s}f\left(\v u+t\v v\right)\, dt\,\phi\left(\v u+s\v v\right)\, ds\, d\v u.\qedhere
\]
\end{proof}
\begin{prop}
For $w\in\mathcal{E}'\left(\mathbb{R}^{n}\right)$:
\[
\mathcal{D}_{\v v}\mathcal{I}_{\v v}w=w,\quad\mathcal{I}_{\v v}\mathcal{D}_{\v v}w=w.
\]
\end{prop}
\begin{proof}
We first observe that for $\phi\in\mathscr{S}\left(\mathbb{R}^{n}\right)$,
$\mathcal{X}_{\v v}\mathcal{D}_{\v v}\phi=0$, and so:
\[
\mathcal{I}_{\v v}\mathcal{D}_{\v v}\phi\left(\v u+t\v v\right)=\int_{-\infty}^{t}\phi\left(\v u+s\v v\right)\, ds=\phi\left(\v u+t\v v\right).
\]
Then:
\begin{flalign*}
 & \hfill & \left\langle \mathcal{D}_{\v v}\mathcal{I}_{\v v}w,\phi\right\rangle  & =-\left\langle \mathcal{I}_{\v v}w,\mathcal{D}_{\v v}\phi\right\rangle  & \hfill\\
 &  &  & =\left\langle w,\mathcal{I}_{\v v}\mathcal{D}_{\v v}\phi\right\rangle \\
 &  &  & =\left\langle w,\phi\right\rangle .
\end{flalign*}
On the other hand,
\[
\mathcal{D}_{\v v}\mathcal{I}_{\v v}\phi\left(\v u+t\v v\right)=\phi\left(\v u+s\v v\right)-\mathcal{X}_{\v v}\phi\otimes\psi_{0}\left(\v u+t\v v\right),
\]
and since $\mathcal{X}_{\v v}\phi\otimes\psi_{0}$ is supported away
from the support of $w$, a similar computation also yields:
\[
\mathcal{I}_{\v v}\mathcal{D}_{\v v}w=w.\qedhere
\]
\end{proof}
\begin{prop}
Let $U\subseteq\mathbb{R}^{n}$ be open and assume $U$ is invariant
under translation in the direction $-\v v$. That is, $U-t\v v\subseteq U$
for $t\ge0$. If $w_{1},w_{2}\in\mathcal{E}'\left(\mathbb{R}^{n}\right)$,
and are equal on $U$, then $\mathcal{I}_{\v v}w_{1}=\mathcal{I}_{\v v}w_{2}$
are equal on $U$.\end{prop}
\begin{proof}
Let $\phi$ be supported in $U$. Then if we inspect:
\[
\mathcal{I}_{\v v}\phi\left(\v u+t\v v\right)=\int_{-\infty}^{t}\phi\left(\v u+s\v v\right)-\mathcal{X}_{\v v}\phi\otimes\psi_{0}\left(\v u+s\v v\right)\, ds,
\]
Then for $\v u+t\v v\notin U$, we must have $\v u+\tau\v v\notin U$
for $\tau\ge t$ and so:
\begin{flalign*}
 & \hfill & \mathcal{I}_{\v v}\phi\left(\v u+t\v v\right) & =\int_{-\infty}^{\infty}\phi\left(\v u+s\v v\right)\, ds-\mathcal{X}_{\v v}\phi\left(\v u\right)\int_{-\infty}^{t}\psi_{0}\left(s\right)\, ds & \hfill\\
 &  &  & =\mathcal{X}_{\v v}\phi\left(\v u\right)\int_{t}^{\infty}\psi_{0}\left(s\right)\, ds.
\end{flalign*}
In particular, $\mathcal{I}_{\v v}\phi\left(\v u+t\v v\right)=0$
when $\v u+t\v v\notin U$ and\emph{ $t\ge t_{\min}$. }In particular:
\[
\mathrm{supp}\,\mathcal{I}_{\v v}\phi\subseteq U\cup V,\quad V=\left\{ \v u+t\v v:t<t_{\min}\right\} .
\]
We now choose a partiton of unity $\rho_{U}$ and $\rho_{V}$ for
$U$ and $V$ so that $\rho_{U}+\rho_{V}=1$ on $U\cup V$, $\mathrm{supp}\,\rho_{U}\subseteq U$,
and $\mathrm{supp}\,\rho_{V}\subseteq V$. Then 
\begin{flalign*}
 & \hfill & \left\langle \mathcal{I}_{\v v}w_{1},\phi\right\rangle  & =-\left\langle w_{1},\mathcal{I}_{\v v}\phi\right\rangle  & \hfill\\
 &  &  & =-\left\langle w_{1},\rho_{U}\mathcal{I}_{\v v}\phi\right\rangle -\underbrace{\left\langle w_{1},\rho_{V}\mathcal{I}_{\v v}\phi\right\rangle }_{=0}\\
 &  &  & =-\left\langle w_{2},\rho_{U}\mathcal{I}_{\v v}\phi\right\rangle -\underbrace{\left\langle w_{2},\rho_{V}\mathcal{I}_{\v v}\phi\right\rangle }_{=0}\\
 &  &  & =-\left\langle w_{2},\mathcal{I}_{\v v}\phi\right\rangle \\
 &  &  & =\left\langle \mathcal{I}_{\v v}w_{2},\phi\right\rangle . & \qedhere
\end{flalign*}
\end{proof}
\begin{prop}
\label{thm:symmetry}Let $t_{0}>t_{\max}=\max_{\v x\in\mathrm{supp}\, w}\v x\cdot\v v$,
then define $w^{\star}$ by:
\[
\left\langle w^{\star},\phi\right\rangle =\left\langle w,\phi\left(\v x+t\vec{v}\right)-\phi\left(\v x+\left(2t_{0}-t\right)\v v\right)\right\rangle .
\]
Then $w^{\star}$ and $\mathcal{I}_{\v v}w^{\star}$ are distribution
with odd and even symmetries across the hyperplane $\left\{ t=t_{0}\right\} $,
and furthermore, $\mathcal{I}_{\v v}w^{\star}$ is compactly supported,
and is equal to $\mathcal{I}_{\v v}w$ on $\left\{ t<t_{0}\right\} $.\end{prop}
\begin{proof}
We observe
\begin{flalign*}
 & \hfill & \left\langle w^{\star},\phi\left(\v u+\left(2t_{0}-t\right)\v v\right)\right\rangle  & =\left\langle w,\phi\left(\v u+\left(2t_{0}-t\right)\vec{v}\right)-\phi\left(\v u+t\v v\right)\right\rangle  & \hfill\\
 &  &  & =-\left\langle w,\phi\left(\v u+t\v v\right)-\phi\left(\v u+\left(2t_{0}-t\right)\vec{v}\right)\right\rangle \\
 &  &  & =-\left\langle w^{\star},\phi\left(\v u+t\v v\right)\right\rangle .
\end{flalign*}
This implies that if $\phi$ has even symmetry across $\left\{ t=t_{0}\right\} $,
then $\left\langle w^{\star},\phi\right\rangle =0$. Then
\begin{flalign*}
 & \hfill &  & \left\langle \mathcal{I}_{\v v}w^{\star},\phi\left(\v u+\left(2t_{0}-t\right)\v v\right)\right\rangle  & \hfill\\
 &  &  & \hspace{3em}=-\left\langle w^{\star},\mathcal{I}_{\v v}\left\{ \phi\left(\v u+\left(2t_{0}-t\right)\vec{v}\right)\right\} \right\rangle \\
 &  &  & \hspace{3em}=-\left\langle w^{\star},\int_{-\infty}^{t}\phi\left(\v u+\left(2t_{0}-s\right)\v v\right)-\mathcal{X}_{\v v}\phi\otimes\psi_{0}\left(\v u+t\v v\right)\, ds\right\rangle \\
 &  &  & \hspace{3em}=-\left\langle w^{\star},\int_{2t_{0}-t}^{\infty}\phi\left(\v u+s\v v\right)-\mathcal{X}_{\v v}\phi\otimes\psi_{0}\left(\v u+\left(2t_{0}-s\right)\v v\right)\, ds\right\rangle \\
 &  &  & \hspace{3em}=\left\langle w^{\star},\int_{-\infty}^{2t_{0}-t}\phi\left(\v u+s\v v\right)-\mathcal{X}_{\v v}\phi\otimes\psi_{0}\left(\v u+\left(2t_{0}-s\right)\v v\right)\, ds\right\rangle \\
 &  &  & \hspace{3em}=\left\langle w^{\star},\int_{t}^{\infty}\phi\left(\v u+s\v v\right)-\mathcal{X}_{\v v}\phi\otimes\psi_{0}\left(\v u+\left(2t_{0}-s\right)\v v\right)\, ds\right\rangle \\
 &  &  & \hspace{3em}=-\left\langle w^{\star},\int_{-\infty}^{t}\phi\left(\v u+s\v v\right)-\mathcal{X}_{\v v}\phi\left(\v u\right)\left(\psi_{0}\left(s\right)-\psi_{0}\left(s\right)+\psi_{0}\left(2t_{0}-s\right)\right)\, ds\right\rangle \\
 &  &  & \hspace{3em}=-\left\langle w^{\star},\mathcal{I}_{\v v}\phi\right\rangle +\left\langle w^{\star},\mathcal{X}_{\v v}\phi\left(\v u\right)\int_{-\infty}^{t}\psi_{0}\left(s\right)-\psi_{0}\left(2t_{0}-s\right)\, ds\right\rangle .
\end{flalign*}
We then observe that $\int_{-\infty}^{t}\psi_{0}\left(s\right)-\psi_{0}\left(2t_{0}-s\right)\, ds$
has even symmetry across $\left\{ t=t_{0}\right\} $, and so:
\[
\left\langle \mathcal{I}_{\v v}w^{\star},\phi\left(\v x+\left(2t_{0}-t\right)\v v\right)\right\rangle =\left\langle \mathcal{I}_{\v v}w^{\star},\phi\right\rangle .
\]
The symmetry easily implies that $\mathcal{I}_{\v v}w^{\star}$ must
be compactly-supported. Furthermore, it is clear that $w=w^{\star}$
on $\left\{ t<t_{0}\right\} $, so we have $\mathcal{I}_{\v v}w=\mathcal{I}_{\v v}w^{\star}$
on $\left\{ t<t_{0}\right\} $.
\end{proof}
We will now establish a relationship between $WF\left(w\right)$ and
$WF\left(\mathcal{I}_{\v v}w\right)$. But first, we start with the
following theorem:
\begin{thm}[Microlocal property\cite{hormander1990AnaLinParDifOpeIDisTheFouAna}]
\label{thm:microlocal}If $P$ is a differential operator of order
$m$ with $\mathcal{C}^{\infty}$ coefficients on a manifold $X$,
thenLet $u\in\mathcal{S}'$ and $a\in S^{\infty}$; then
\[
WF\left(u\right)\subseteq\mathrm{Char\,}P\cup WF\left(Pu\right),\quad u\in\mathscr{D}'\left(X\right),
\]
where the characteristic set $\mathrm{Char\,}P$ is defined by
\[
\mathrm{Char}\, P=\left\{ \left(\v x,\v{\xi}\right)\in T^{\star}\left(X\right)\,\middle|\, P_{m}\left(\v x,\v{\xi}\right)=0\right\} =\mathbb{R}^{n}\times\v v^{\perp},
\]
and $P_{m}$ is the principal symbol of $P$.
\end{thm}
In the case that $P$ is merely a directional derivative, i.e., $P=\mathcal{D}_{\v v}$,
then $P_{m}\left(\v x,\v{\xi}\right)=i\v{\xi}\cdot\v v$, and so
\[
\mathrm{Char}\, P=\mathbb{R}^{n}\times\v v^{\perp}.
\]

While \ref{thm:microlocal} implies that $\mathcal{I}_{\v v}$ will
extend the wavefront set of a distriution by \textbf{\emph{at most}}
$\mathbb{R}^{n}\times\left(\v v^{\perp}\backslash\v 0\right)$, the
following result states that $\mathcal{I}_{\v v}w$ \textbf{\emph{will
not}} contain an element $\left(\v x,\v{\eta}_{0}\right)\in\mathbb{R}^{n}\times\v v^{\perp}$
in its wavefront set if $WF\left(w\right)$ omits $\mathbb{R}^{n}\times\left\{ \v{\eta}_{0}\right\} $
altogether.
\begin{thm}[Main Result]
\label{thm:mainresult}Let $w\in\mathcal{E}'\left(\mathbb{R}^{n}\right)$,
and $\v{\eta}_{0}\in\v v^{\perp}$. If:
\[
WF\left(w\right)\cap\left(\mathbb{R}^{n}\times\left\{ \v{\eta}_{0}\right\} \right)=\emptyset,
\]
then:
\[
WF\left(\mathcal{I}_{\v v}w\right)\cap\left(\mathbb{R}^{n}\times\left\{ \v{\eta}_{0}\right\} \right)=\emptyset.
\]
\end{thm}
\begin{proof}
Let $w^{\star}$ be the distribution extending $w$, having odd symmetry
across a plane $\left\{ t=t_{0}\right\} $, with $t_{0}$ large enough
so that $w=w^{\star}$ on $\left\{ t<t_{0}\right\} $. (\ref{thm:symmetry}).
Then:
\[
WF\left(w^{\star}\right)\cap\left(\mathbb{R}^{n}\times\left\{ \v{\eta}_{0}\right\} \right)=\emptyset,
\]
and it will suffice to show that:
\[
WF\left(\mathcal{I}_{\v v}w^{\star}\right)\cap\left(\mathbb{R}^{n}\times\left\{ \v{\eta}_{0}\right\} \right)=\emptyset.
\]

Indeed, since $\mathcal{I}_{\v v}w^{\star}$ and $w^{\star}$ are
compactly-supported distribution, their Fourier transforms exist as
entire functions, and:
\[
i\tau\widehat{\mathcal{I}_{\v v}w^{\star}}\left(\v{\eta}+\tau\v v\right)=\widehat{w^{\star}}\left(\v{\eta}+\tau\v v\right),\quad\v{\eta}\in\v v^{\perp},\tau\in\mathbb{R}.
\]

Because of the odd symmetry of $w^{\star}$ across the plane $\left\{ t=t_{0}\right\} $,
$\widehat{w^{\star}}$ vanishes on $\v v^{\perp}$, and also:
\[
\widehat{\mathcal{I}_{\v v}w^{\star}}\left(\v{\eta}+\tau\v v\right)=\begin{cases}
\frac{\widehat{w^{\star}}\left(\v{\eta}+\tau\v v\right)}{i\tau}, & \text{if }\tau\ne0,\\
-i\mathcal{D}_{\v v}\widehat{w^{\star}}\left(\v{\eta}\right), & \text{if }\tau=0,
\end{cases}
\]
for $\v{\eta}\in\v v^{\perp}$ and $\tau\in\mathbb{R}$. The case
$\tau=0$ comes from an application of l\textquoteright Hôpital\textquoteright s
rule. However, $-i\mathcal{D}_{\v v}\widehat{w^{\star}}=-\widehat{tw^{\star}}$,
and:
\[
WF\left(-tw^{\star}\right)\subseteq WF\left(w^{\star}\right),
\]
so we can expect $-i\mathcal{D}_{\v v}\widehat{w^{\star}}\left(\v{\xi}\right)$
to decay rapidly in an open conic neighborhood $\Gamma$ of $\v{\eta}_{0}$.
For our following argument, we will require that $\Gamma$ be chosen
to be convex in $\tau$, i.e., if $\v{\eta}+\tau_{k}\v v\in\Gamma$
for $k=1,2$, and $\tau_{1}<\tau_{2}$, then $\v{\eta}+\tau\v v\in\Gamma$
for $\tau_{1}\le\tau\le\tau_{2}$. In particular, we can set:
\[
\Gamma=\left\{ \sigma\v{\eta}_{0}+\v{\xi}+\tau\v v:\sigma,\tau>0,\v{\xi}\in\v v^{\perp}\cap\v{\eta}_{0}^{\perp},\max\left\{ \norm{\xi},\left|\tau\right|\right\} <\varepsilon\sigma\right\} ,
\]
for some $\varepsilon>0$ sufficiently small. Then for each $N\ge0$,
we can choose $C_{N}$ so that:
\[
\left|\mathcal{D}_{\v v}\widehat{w^{\star}}\left(\v{\eta}+\tau\v v\right)\right|\le C_{N}\left(1+\norm{\v{\eta}}^{2}+\tau^{2}\right)^{-\nicefrac{N}{2}},\quad\v{\eta}\in\v v^{\perp},\tau\in\mathbb{R},\v{\eta}+\tau\v v\in\Gamma.
\]
This bound on the derivative then gives us the following bound on
$\widehat{w^{\star}}$:
\[
\left|\widehat{w^{\star}}\left(\v{\eta}+\tau\v v\right)\right|\le C_{N}\left(1+\norm{\v{\eta}}^{2}+\tau^{2}\right)^{-\nicefrac{N}{2}}\left|\tau\right|,\quad\v{\eta}\in\v v^{\perp},\tau\in\mathbb{R},\v{\eta}+\tau\v v\in\Gamma.
\]
It then follows immediately that:
\[
\left|\widehat{\mathcal{I}_{\v v}w^{\star}}\left(\v{\eta}+\tau\v v\right)\right|\le C_{N}\left(1+\norm{\v{\eta}}^{2}+\tau^{2}\right)^{-\nicefrac{N}{2}},\quad\v{\eta}\in\v v^{\perp},\tau\in\mathbb{R},\v{\eta}+\tau\v v\in\Gamma.
\]
As this is true for arbitrary $N\ge0$, this proves that:
\[
WF\left(\mathcal{I}_{\v v}w^{\star}\right)\cap\left(\mathbb{R}^{n}\times\left\{ \v{\eta}_{0}\right\} \right)=\emptyset.
\]
Since $\mathcal{I}_{\v v}w=\mathcal{I}_{\v v}w^{\star}$ on $\left\{ t<t_{0}\right\} $,
we can say that:
\[
WF\left(\mathcal{I}_{\v v}w\right)\cap\left(\left\{ t<t_{0}\right\} \times\left\{ \v{\eta}_{0}\right\} \right)=\emptyset,
\]
and then let $t_{0}\rightarrow\infty$ to obtain:
\[
WF\left(\mathcal{I}_{\v v}w\right)\cap\left(\mathbb{R}^{n}\times\left\{ \v{\eta}_{0}\right\} \right)=\emptyset.\qedhere
\]

\end{proof}
The above result implies that the only way $\mathcal{I}_{\v v}w$
can include an element $\left(\v x_{0},\v{\eta}_{0}\right)\in\mathbb{R}^{n}\times\v v^{\perp}$
in its wavefront set is if $w$ itself already contains some element
of $\mathbb{R}^{n}\times\left\{ \v{\eta}_{0}\right\} $. The following
result further refines the previous result by describing a necessary
condition on $\v x_{0}$ in order for $\mathcal{I}_{\v v}w$ to contain
$\left(\v x_{0},\v{\eta}_{0}\right)$ in its wavefront set. Intuition
tells us that $\left(\v x_{0}-t\v v,\v{\eta}_{0}\right)$ must already
belong to the wavefront set of $w$ for some $t\ge0$.

\newpage{}
\begin{prop}
\label{thm:propagation1}Let $U\subseteq\mathbb{R}^{n}$ be open and
assume $U$ is invariant under translation in the direction $-\v v$,
and $\v{\eta}_{0}\in\v v^{\perp}$. If:
\begin{equation}
WF\left(w\right)\cap\left(U\times\left\{ \v{\eta}_{0}\right\} \right)=\emptyset,\label{eq:cond}
\end{equation}
then:
\[
WF\left(\mathcal{I}_{\v v}w\right)\cap\left(U\times\left\{ \v{\eta}_{0}\right\} \right)=\emptyset.
\]
\end{prop}
\begin{proof}
Let $U^{\star}$ be an open subset of $U$ whose closure is entirely
contained in $U$, and is also closed under translation in the direction
$-\v v$, and choose a $\mathcal{C}^{\infty}$ function $\psi\ge0$
supported in $U$ that is equal to 1 on $U^{\star}$. Then $\psi w=w$
on $U^{\star}$, so $\mathcal{I}_{\v v}\left(\psi w\right)=\mathcal{I}_{\v v}w$,
and in fact, we also have $\psi\mathcal{I}_{\v v}w=\mathcal{I}_{\v v}w$
on $U^{\star}$. Furthermore, \ref{eq:cond} implies:
\[
WF\left(\psi w\right)\cap\left(\mathbb{R}^{n}\times\left\{ \v{\eta}_{0}\right\} \right)=\emptyset,
\]
and so:
\[
WF\left(\mathcal{I}_{\v v}\left(\psi w\right)\right)\cap\left(\mathbb{R}^{n}\times\left\{ \v{\eta}_{0}\right\} \right)=\emptyset.
\]
We then have:
\begin{flalign*}
 & \hfill & WF\left(\mathcal{I}_{\v v}w\right)\cap\left(U^{\star}\times\left\{ \v{\eta}_{0}\right\} \right) & =WF\left(\mathcal{I}_{\v v}\left(\psi w\right)\right)\cap\left(U^{\star}\times\left\{ \v{\eta}_{0}\right\} \right) & \hfill\\
 &  &  & \subseteq WF\left(\mathcal{I}_{\v v}\left(\psi w\right)\right)\cap\left(\mathbb{R}^{n}\times\left\{ \v{\eta}_{0}\right\} \right)\\
 &  &  & =\emptyset.
\end{flalign*}
Since $U^{\star}$ was arbitrary, and for each $\v x\in U$, we can
find such a $U^{\star}$ containing $\v x$, we can deduce that:
\[
WF\left(\mathcal{I}_{\v v}w\right)\cap\left(U\times\left\{ \v{\eta}_{0}\right\} \right)=\emptyset.\qedhere
\]
\end{proof}
\begin{cor}[Propagation of singularities of the distributional directional antiderivative]
\label{thm:propagation-ray}Let $w\in\mathcal{E}'\left(\mathbb{R}^{n}\right)$,
define:
\begin{equation}
R_{\v v}\left(\v x\right)=\left\{ \v x+t\v v:t\ge0\right\} ,\quad\v x\in\mathbb{R}^{n},\v v\in\mathcal{S}^{n-1},\label{eq:ray}
\end{equation}
 and let:
\[
V_{\v{\eta}_{0}}=\bigcup_{\left(\v x,\v{\eta}_{0}\right)\in WF\left(w\right)}R_{\v v}\left(\v x\right),\quad U_{\v{\eta}_{0}}=V_{\v{\eta}_{0}}^{C},\quad\v{\eta}_{0}\in\v v^{\perp}.
\]
Then:
\begin{flalign}
 & \hfill & WF\left(\mathcal{I}_{\v v}w\right) & \subseteq WF\left(w\right)\cup\bigcup_{\v{\eta}_{0}\in\v v^{\perp}}\left(V_{\v{\eta}_{0}}\times\left\{ \v{\eta}_{0}\right\} \right) & \hfill\nonumber \\
 &  &  & =WF\left(w\right)\cup\left\{ \left(\v x+t\v v,\v{\eta}_{0}\right)\,\middle|\,\left(\v x,\v{\eta}_{0}\right)\in WF\left(w\right),\v{\eta}_{0}\perp\v v,t\ge0\right\} .\label{eq:propagation}
\end{flalign}
\end{cor}
\begin{proof}
Since $\mathcal{D}_{\v v}\mathcal{I}_{\v v}w=w$, we can already narrow
down $WF\left(\mathcal{I}_{\v v}w\right)$ to:
\begin{equation}
WF\left(\mathcal{I}_{\v v}w\right)\subseteq WF\left(w\right)\cup\left(\mathbb{R}^{n}\times\v v^{\perp}\right).\label{eq:xavierresult}
\end{equation}
We want to be able to replace $\mathbb{R}^{n}\times\v v^{\perp}$
with $\bigcup_{\v{\eta}_{0}\in\v v^{\perp}}\left(V_{\v{\eta}_{0}}\times\left\{ \v{\eta}_{0}\right\} \right)$.

We next observe that for each $\v{\eta}_{0}\in\v v^{\perp}$, since
$w$ is compactly-supported, $V_{\v{\eta}_{0}}$ must be closed. Then
$U_{\v{\eta}_{0}}$ is an open set that is invariant under translation
in the direction of $-\v v$, and:
\[
WF\left(w\right)\cap\left(U_{\v{\eta}_{0}}\times\left\{ \v{\eta}_{0}\right\} \right)=\emptyset,
\]
and so:
\begin{equation}
WF\left(\mathcal{I}_{\v v}w\right)\cap\left(U_{\v{\eta}_{0}}\times\left\{ \v{\eta}_{0}\right\} \right)=\emptyset.\label{eq:WFIw}
\end{equation}
Therefore, if $\left(\v x_{0},\v{\eta}_{0}\right)\in WF\left(\mathcal{I}_{\v v}w\right)$,
but $\left(\v x_{0},\v{\eta}_{0}\right)\notin WF\left(w\right)$,
then \ref{eq:xavierresult} indicates that $\v{\eta}_{0}\in\v v^{\perp}$,
and then \ref{eq:WFIw} would require that $\v x\notin U_{\v{\eta}_{0}}$,
and so $\left(\v x_{0},\v{\xi}_{0}\right)\in V_{\v{\eta}_{0}}\times\left\{ \v{\eta}_{0}\right\} $.

Note that if we lift the restriction of compact support on $w$, we
must instead use:
\[
V_{\v{\eta}_{0}}=\overline{\bigcup_{\left(\v x,\v{\eta}_{0}\right)\in WF\left(w\right)}R_{\v v}\left(\v x\right)}.\qedhere
\]

\end{proof}
We now strengthen \ref{thm:propagation1} by with the next result.
Intuitively, even if $\left(\v x_{0},\v{\eta}_{0}\right)\in WF\left(\mathcal{I}_{\v v}w\right)$,
if $WF\left(\mathcal{I}_{\v v}w\right)$ omits some $\left(\v x_{0}+t_{0}\v v,\v{\eta}_{0}\right)$,
for some $t_{0}>0$, then the only way for $WF\left(\mathcal{I}_{\v v}w\right)$
to pick up any more elements of the form $\left(\v x_{0}+t\v v,\v{\eta}_{0}\right)$
for $t>t_{0}$ is for $WF\left(w\right)$ to contain some $\left(\v x_{0}+t_{1}\v v,\v{\eta}_{0}\right)$,
for some $t_{1}>t_{0}$.
\begin{prop}
Let $w\in\mathcal{E}'\left(\mathbb{R}^{n}\right)$, $\v{\eta}_{0}\in\v v^{\perp}$,
and $U_{0}$ be a bounded open set, $t_{1}>0$, and 
\[
U_{t}=U+t\v v,\quad t\in\mathbb{R},
\]
\[
U=\bigcup_{0\le t\le t_{1}}U_{t}.
\]
If:
\[
WF\left(\mathcal{I}_{\v v}w\right)\cap\left(U_{0}\times\left\{ \v{\eta}_{0}\right\} \right)=\emptyset,
\]
and:
\[
WF\left(w\right)\cap\left(U\times\left\{ \v{\eta}_{0}\right\} \right)=\emptyset,
\]
then:
\[
WF\left(\mathcal{I}_{\v v}w\right)\cap\left(U\times\left\{ \v{\eta}_{0}\right\} \right)=\emptyset.
\]
\end{prop}
\begin{proof}
We may assume without loss of generality that $U_{0}$ is convex,
otherwise, we can apply the following argument to every convex open
subset of $U_{0}$. Let $U_{0}^{\star}$ be an open set whose closure
is contained in $U_{0}$, then define $U^{\star}$ in much the same
way as $U$. Now let $\psi\in\mathcal{C}_{0}^{\infty}\left(U\right)$
be equal to 1 on $U^{\star}$. Consider the distributional partial
derivative:
\[
\mathcal{D}_{\v v}\left(\psi\cdot\mathcal{I}_{\v v}w\right)=\mathcal{D}_{\v v}\psi\cdot\mathcal{I}_{\v v}w+\psi\cdot w.
\]
Since $\mathcal{D}_{\v v}\psi$ vanishes on $U^{\star}$, we have
$\mathcal{D}_{\v v}\left(\psi\cdot\mathcal{I}_{\v v}w\right)=w$ on
$U^{\star}$, and since:
\[
WF\left(w\right)\cap\left(U^{\star}\times\left\{ \v{\eta}_{0}\right\} \right)\subseteq WF\left(w\right)\cap\left(U\times\left\{ \v{\eta}_{0}\right\} \right)=\emptyset,
\]
we have:
\[
WF\left(\mathcal{D}_{\v v}\left(\psi\cdot\mathcal{I}_{\v v}w\right)\right)\cap\left(U^{\star}\times\left\{ \v{\eta}_{0}\right\} \right)=\emptyset.
\]
Furthermore, since:
\[
WF\left(\mathcal{I}_{\v v}w\right)\cap\left(U_{0}\times\left\{ \v{\eta}_{0}\right\} \right)=\emptyset,
\]
it must also follow that:
\[
WF\left(\mathcal{D}_{\v v}\left(\psi\cdot\mathcal{I}_{\v v}w\right)\right)\cap\left(U_{0}\times\left\{ \v{\eta}_{0}\right\} \right)=\emptyset,
\]
and so:
\[
WF\left(\mathcal{D}_{\v v}\left(\psi\cdot\mathcal{I}_{\v v}w\right)\right)\cap\left(\left(U_{0}\cup U^{\star}\right)\times\left\{ \v{\eta}_{0}\right\} \right)=\emptyset.
\]
We can replace $U_{0}\cup U^{\star}$ with $\bigcup_{t\le0}U_{t}\cup U^{\star}$,
as that introduces no points that are inside the support of $\psi$
(hence the requirement that $U_{0}$ be convex), and is also invariant
under translation in the direction $-\v v$, and so we again apply
\ref{thm:mainresult} to obtain:
\[
WF\left(\psi\cdot\mathcal{I}_{\v v}w\right)\cap\left(\left(U_{0}\cup U^{\star}\right)\times\left\{ \v{\eta}_{0}\right\} \right)=\emptyset.
\]
In particular, 
\[
WF\left(\mathcal{I}_{\v v}w\right)\cap\left(U^{\star}\times\left\{ \v{\eta}_{0}\right\} \right)=\emptyset,
\]
and since $U_{0}^{\star}$ was arbitrary, we can replace $U^{\star}$
with $U$:
\[
WF\left(\mathcal{I}_{\v v}w\right)\cap\left(U\times\left\{ \v{\eta}_{0}\right\} \right)=\emptyset.\qedhere
\]

\end{proof}
We now wish to extend the distributional directional antiderivative
further by replacing the support condition \ref{eq:support-condition}
with an even weaker condition, that there exists a $t_{\min}\in\mathcal{C}^{\infty}\left(\v v^{\perp}\right)$
such that
\begin{equation}
\mathrm{supp}\, w\subseteq\left\{ \v u+t\v v\,\middle|\,\v u\in\v v^{\perp},t>t_{\min}\left(\v u\right)\right\} .\label{eq:support-condition-2}
\end{equation}
Notice this includes the previous support condition by considering
the case that $t_{\min}$is a constant. If we let $\chi\left(\v x\right)=\v x-t_{\min}\left(\v u\right)\v v$,
then the pullback $\chi^{\star}w$ has the support condition $\inf_{\v x\in\mathrm{supp}\,\chi^{\star}w}\v x\cdot\v v>0$,
and so we can define $\mathcal{I}_{\v v}w$ by conjugating $\mathcal{I}_{\v v}$
with the pullback map $\chi^{\star}$. We will want to be sure that
this does not change $\mathcal{I}_{\v v}w$, however.
\begin{prop}
Let $w$ satisfy \ref{eq:support-condition}, and define $\chi\left(\v x\right)=\v x-t_{\min}\v v$.
Then $\chi^{-\star}\mathcal{I}_{\v v}\chi^{\star}w=\mathcal{I}_{\v v}w$.\end{prop}
\begin{proof}
It is important to note that $\mathcal{X}_{\v v}\chi^{\star}=\mathcal{X}_{\v v}$.
For some $\psi_{0}\in\mathcal{C}_{0}^{\infty}\left(\mathbb{R}^{-}\right)$,
we have
\begin{flalign*}
 & \hfill & \left\langle \chi^{-\star}\mathcal{I}_{\v v}\chi^{\star}w,\phi\right\rangle  & =-\left\langle \chi^{\star}w,\mathcal{I}_{\v v}\chi^{\star}\phi\right\rangle  & \hfill\\
 &  &  & =-\left\langle \chi^{\star}w,\int_{-\infty}^{t}\chi^{\star}\phi\left(\v u+s\v v\right)-\mathcal{X}_{\v v}\left(\chi^{\star}\phi\right)\otimes\psi_{0}\left(\v u+s\v v\right)\, ds\right\rangle \\
 &  &  & =-\left\langle \chi^{\star}w,\int_{-\infty}^{t}\phi\left(\v u+\left(s-t_{\min}\right)\v v\right)-\mathcal{X}_{\v v}\phi\otimes\psi_{0}\left(\v u+s\v v\right)\, ds\right\rangle \\
 &  &  & =-\left\langle \chi^{\star}w,\int_{-\infty}^{t-t_{\min}}\phi\left(\v u+s\v v\right)-\mathcal{X}_{\v v}\phi\otimes\psi_{0}\left(\v u+\left(s+t_{\min}\right)\v v\right)\, ds\right\rangle \\
 &  &  & =-\left\langle \chi^{\star}w,\int_{-\infty}^{t-t_{\min}}\phi\left(\v u+s\v v\right)-\mathcal{X}_{\v v}\phi\otimes\chi^{-\star}\psi_{0}\left(\v u+s\v v\right)\, ds\right\rangle \\
 &  &  & =-\left\langle \chi^{\star}w,\mathcal{I}_{\v v}\phi\left(\v u+\left(t-t_{\min}\right)\v v\right)\right\rangle \\
 &  &  & =-\left\langle w,\chi^{-\star}\left\{ \mathcal{I}_{\v v}\phi\left(\v u+\left(t-t_{\min}\right)\v v\right)\right\} \right\rangle \\
 &  &  & =-\left\langle w,\mathcal{I}_{\v v}\phi\left(\v u+t\v v\right)\right\rangle \\
 &  &  & =\left\langle \mathcal{I}_{\v v}w,\phi\right\rangle ,
\end{flalign*}
where $\chi^{-\star}\psi_{0}\left(t\right)=\psi_{0}\left(t+t_{\min}\right)$.
\end{proof}
Thus, we may define $\mathcal{I}_{\v v}w$ when $w$ satisfies the
weaker support condition as follows:
\begin{defn}
\label{def:Ivw2}Let $w\in\mathscr{D}'\left(\mathbb{R}^{n}\right)$
have the support condition \ref{eq:support-condition-2}. Define
\[
\mathcal{I}_{\v v}w=\chi^{-\star}\mathcal{I}_{\v v}\chi^{\star}w,
\]
where $\chi\left(\v x\right)=\v x-t_{\min}\left(\v u\right)\v v$.
\end{defn}
We can also verify that this definition is independent of the choice
of $\chi$ so long as \ref{eq:support-condition-2} is satisfied.
We will omit this proof as it would proceed in a fashion similar to
the above computation.

We also wish to show \ref{thm:propagation-ray} also applies to this
extension of $\mathcal{I}_{\v v}$.
\begin{proof}
Notice that
\[
\mathcal{D}\chi=\begin{bmatrix}I_{\v v^{\perp}} & \v 0\\
\mathcal{D}t_{\min} & 1
\end{bmatrix},
\]
which implies that $\mathcal{D}\chi^{T}$ and $\mathcal{D}\chi^{-T}$
both fix $\v v^{\perp}$. Thus,

\begin{flalign*}
 & \hfill & WF\left(\mathcal{I}_{\v v}w\right) & =WF\left(\chi^{-\star}\mathcal{I}_{\v v}\chi^{\star}w\right) & \hfill\\
 &  &  & =\chi^{-\star}WF\left(\mathcal{I}_{\v v}\chi^{\star}w\right)\\
 &  &  & \subseteq\chi^{-\star}\left(WF\left(\chi^{\star}w\right)\cup\left\{ \left(\v x+t\v v,\v{\eta}_{0}\right)\,\middle|\,\right.\right.\\
 &  &  & \hspace{3em}\left.\left.\left(\v x,\v{\eta}_{0}\right)\in WF\left(\chi^{\star}w\right),\v{\eta}_{0}\perp\v v,t\ge0\right\} \right)\\
 &  &  & =WF\left(w\right)\cup\chi^{-\star}\left\{ \left(\v x+t\v v,\v{\eta}_{0}\right)\,\middle|\,\right.\\
 &  &  & \hspace{3em}\left.\left(\v x,\v{\eta}_{0}\right)\in WF\left(\chi^{\star}w\right),\v{\eta}_{0}\perp\v v,t\ge0\right\} \\
 &  &  & =WF\left(w\right)\cup\left\{ \left(\chi^{-1}\left(\v x+t\v v\right),\v{\eta}_{0}\right)\,\middle|\,\right.\\
 &  &  & \hspace{3em}\left.\left(\chi\left(\v x\right),\v{\eta}_{0}\right)\in WF\left(w\right),\v{\eta}_{0}\perp\v v,t\ge0\right\} \\
 &  &  & =WF\left(w\right)\cup\left\{ \left(\v x+t\v v,\v{\eta}_{0}\right)\,\middle|\,\left(\v x,\v{\eta}_{0}\right)\in WF\left(w\right),\v{\eta}_{0}\perp\v v,t\ge0\right\} . & \qedhere
\end{flalign*}
\end{proof}
\begin{defn}
\label{def:DistIntegral-Notation}We will use the more familiar notation:
\begin{equation}
\int_{0}^{\infty}w\left(\v x-t\v v\right)\, dt\label{eq:DistIntegral}
\end{equation}
to refer to $\mathcal{I}_{\v v}w$.\end{defn}

\section{General line integrals}

Now that we have given meaning to the integral \ref{eq:DistIntegral},
we now wish to give meaning to the following integral:
\begin{equation}
\int_{0}^{\infty}w\left(\v x-\v{\gamma}\left(t\right)\right)\upsilon\left(t\right)\, dt,\label{eq:curveintegral}
\end{equation}
given a $\v{\gamma}\in\mathcal{C}^{\infty}\left(\left(-\varepsilon,\infty\right);\mathbb{R}^{n}\right)$
and positive-valued $\upsilon\in\mathcal{C}^{\infty}\left(\left(-\varepsilon,\infty\right),\mathbb{R}\right)$,
for some $\varepsilon>0$, with $\v{\gamma}\left(0\right)=\v 0$,
and $\v{\gamma}'\left(t\right)\ne0$ for all $t>-\varepsilon$. It
will also be necessary to impose a support condition that $w\left(\v x-\gamma\left(t\right)\right)$
has bounded support in the variable $t$ to avoid a pathological choice
of $\v{\gamma}$, e.g, a choice of $\v{\gamma}$, that given some
$w\in\mathcal{L}^{1}\left(\mathbb{R}^{n}\right)$, the above integral
may fail to converge for $\v x$ in some open set.

We observe that in the case that $w\in\mathcal{L}^{1}\left(\mathbb{R}^{n}\right)$,
the above integral can be interpreted as
\[
\left.\int_{0}^{\infty}w\left(\v x-\v{\gamma}\left(y+t\right)\right)\, dt\right|_{y=0}.
\]
The notation $w\left(\v x-\v{\gamma}\left(y\right)\right)$ refers
to pulling back $w$ by the map $\chi\left(\v x,y\right)=\v x-\v{\gamma}\left(y\right)$.
We then compute the distributional antiderivative in the direction
$\left(\v 0,-1\right)$, the direction corresponding to the negative
$y$-axis. This antiderivative is then pulled back by the map
\[
\psi_{0}\left(\v x\right)=\left(\v x,0\right).
\]

For ease of notation, we will specialize to the case $\v{\gamma}\left(0\right)=\v 0$.
A result for the general case can be achieved via translations.
\begin{defn}
\label{thm:curveintegral}Let $w$ be a distribution in $\mathbb{R}^{n}$,
and $\v{\gamma}\in\mathcal{C}^{\infty}\left(\left(-\varepsilon,\infty\right);\mathbb{R}^{n}\right)$
a curve for some $\varepsilon>0$, with $\v{\gamma}\left(0\right)=\v 0$
and $\v{\gamma}'\left(t\right)\ne0$, and assume the pullback
\[
\chi^{\star}w=w\left(\v x-\v{\gamma}\left(y\right)\right)
\]
has support bounded in $y$. Then the integral:
\[
\int_{0}^{\infty}w\left(\v x-\v{\gamma}\left(t\right)\right)\, dt,
\]
is defined as:
\[
\left.\int_{0}^{\infty}w\left(\v x-\v{\gamma}\left(y+t\right)\right)\, dt\right|_{y=0}.
\]
That is, 
\[
\int_{0}^{\infty}w\left(\v x-\v{\gamma}\left(t\right)\right)\, dt=\psi_{0}^{\star}\mathcal{I}_{\left(\v 0,-1\right)}\left(\chi^{\star}w\right).
\]

If additionally, $\upsilon\in\mathcal{C}^{\infty}\left(\left(-\varepsilon,\infty\right)\right)$
is a positive-valued weight function, we can then define:
\begin{flalign}
 &  & \int_{0}^{\infty}w\left(\v x-\v{\gamma}\left(t\right)\right)\upsilon\left(t\right)\, dt & =\left.\int_{0}^{\infty}w\left(\v x-\v{\gamma}\left(y+t\right)\right)\upsilon\left(y+t\right)\, dt\right|_{y=0}\label{eq:LineIntegralWeighted}\\
 &  &  & =\psi_{0}^{\star}\mathcal{I}_{\left(\v 0,-1\right)}\left(\left(\chi^{\star}w\right)\upsilon\right).\nonumber 
\end{flalign}

\end{defn}
We then extend \ref{thm:propagation-ray} to a more general result
for the integral \ref{eq:curveintegral}.
\begin{thm}
\label{thm:propagation-curve}If $w\in\mathcal{E}'\left(\mathbb{R}^{n}\right)$
is a distribution such that the support in $y$ of $w\left(\v x-\v{\gamma}\left(y\right)\right)$
is bounded, then:
\begin{flalign*}
 & \hfill &  & WF\left\{ \int_{0}^{\infty}w\left(\v x-\v{\gamma}\left(t\right)\right)\upsilon\left(t\right)\, dt\right\} \backslash WF\left(w\right) & \hfill\\
 &  &  & \hspace{3em}\subseteq\left\{ \left(\v x,\v{\xi}\right)\,\middle|\,\exists t\ge0:\v{\xi}\perp\gamma'\left(t\right)\And\left(\v x-\v{\gamma}\left(t\right),\v{\xi}\right)\in WF\left(w\right)\right\} .
\end{flalign*}
\end{thm}
\begin{proof}
We may, without loss of generality, set $\upsilon=1$, as doing so
will not alter the wavefront sets involved in this proof. With:
\[
\chi\left(\v x,y\right)=\v x-\v{\gamma}\left(y\right).
\]
We observe that:
\[
\mathcal{D}\chi\left(\v x,y\right)=\begin{bmatrix}I_{n} & \gamma'\left(y\right)\end{bmatrix},\quad\mathcal{D}\chi\left(\v x,y\right)^{T}=\begin{bmatrix}I_{n}\\
\gamma'\left(y\right)^{T}
\end{bmatrix},
\]
and so $\ker\mathcal{D}\chi\left(\v x,y\right)^{T}$ is trivial, indicating
that the pullback $w\left(\v x-\v{\gamma}\left(y\right)\right)$ is
indeed well-defined, and: 
\[
WF\left\{ w\left(\v x-\v{\gamma}\left(y\right)\right)\right\} \subseteq\chi^{\star}WF\left(w\right),
\]

Where:
\begin{equation}
\chi^{\star}WF\left(w\right)=\left\{ \left(\left(\v x,y\right),\left(\v{\xi},\gamma'\left(y\right)\cdot\v{\xi}\right)\right):\left(\v x-\v{\gamma}\left(y\right),\v{\xi}\right)\in WF\left(w\right)\right\} .\label{eq:WF1}
\end{equation}
We now wish to show that:
\[
WF\left\{ w\left(\v x-\v{\gamma}\left(y\right)\right)\right\} =\chi^{\star}WF\left(w\right).
\]
Indeed, let $y_{0}>-\varepsilon$, and define:
\[
\psi_{y_{0}}\left(\v x\right)=\begin{bmatrix}\v x+\gamma\left(y_{0}\right)\\
y_{0}
\end{bmatrix}.
\]
Then:
\[
\mathcal{D}_{\psi_{y_{0}}}\left(\v x\right)=\begin{bmatrix}I_{n}\\
0
\end{bmatrix}.\quad\mathcal{D}_{\psi_{y_{0}}}\left(\v x\right)^{T}=\begin{bmatrix}I_{n} & 0\end{bmatrix},
\]
and so:
\[
\ker\mathcal{D}_{\psi_{y_{0}}}\left(\v x\right)^{T}=\left\{ \left(\v 0,\eta\right):\eta\in\mathbb{R}\right\} .
\]
Therefore, the set of normals for $\psi_{y_{0}}$ satisfies:
\[
N_{\psi_{y_{0}}}\subseteq\left(\mathbb{R}^{n}\times\left(-\varepsilon,\infty\right)\right)\times\left\{ \left(\v 0,\eta\right):\eta\in\mathbb{R}\right\} ,
\]
and so $N_{\psi_{y_{0}}}\cap WF\left\{ w\left(\v x-\v{\gamma}\left(y\right)\right)\right\} =\emptyset$.
Thus, the pullback $\psi_{y_{0}}^{\star}\left\{ w\left(\v x-\v{\gamma}\left(y\right)\right)\right\} $
is well-defined, and is in fact equal to $w$. Then:
\begin{equation}
WF\left(w\right)=WF\left(\psi_{y_{0}}^{\star}\left\{ w\left(\v x-\v{\gamma}\left(y\right)\right)\right\} \right)\subseteq\psi_{y_{0}}^{\star}WF\left\{ w\left(\v x-\v{\gamma}\left(y\right)\right)\right\} ,\label{eq:WF2}
\end{equation}
where:
\begin{flalign}
 & \hfill &  & \psi_{y_{0}}^{\star}WF\left\{ w\left(\v x-\v{\gamma}\left(y\right)\right)\right\}  & \hfill\nonumber \\
 &  &  & \hspace{3em}=\left\{ \left(\v x,\v{\xi}\right):\left(\left(\v x+\v{\gamma}\left(y_{0}\right),y_{0}\right),\left(\v{\xi},\eta\right)\right)\in WF\left\{ w\left(\v x-\v{\gamma}\left(y\right)\right)\right\} \right\} .\label{eq:WF2-pullback}
\end{flalign}

Thus, if we chose $\left(\left(\v x,y_{0}\right),\left(\v{\xi},\gamma'\left(y_{0}\right)\cdot\v{\xi}\right)\right)\in\chi^{\star}WF\left(w\right)$,
this choice was based on choosing $\left(\v x-\v{\gamma}\left(y_{0}\right),\v{\xi}\right)\in WF\left(w\right)$.
We then have by \ref{eq:WF2} that $\left(\v x-\v{\gamma}\left(y_{0}\right),\v{\xi}\right)\in\psi_{y_{0}}^{\star}WF\left\{ w\left(\v x-\v{\gamma}\left(y\right)\right)\right\} $.
That is,
\[
\left(\left(\v x,y_{0}\right),\left(\v{\xi},\eta\right)\right)=\left(\left(\v x-\v{\gamma}\left(y_{0}\right)+\v{\gamma}\left(y_{0}\right),y_{0}\right),\left(\v{\xi},\eta\right)\right)\in WF\left\{ w\left(\v x-\v{\gamma}\left(y\right)\right)\right\} ,
\]
for some $\eta$. Of course, one only needs to choose $\eta=\gamma'\left(y_{0}\right)\cdot\v{\xi}$,
and this will yield the desired set inclusion.

Next, we can use \ref{thm:propagation1} to describe $WF\left\{ \int_{0}^{\infty}w\left(\v x-\v{\gamma}\left(y+t\right)\right)\, dt\right\} $
as follows:
\begin{flalign*}
 & \hfill &  & WF\left\{ \int_{0}^{\infty}w\left(\v x-\v{\gamma}\left(y+t\right)\right)\, dt\right\} \backslash WF\left\{ w\left(\v x-\v{\gamma}\left(y\right)\right)\right\}  & \hfill\\
 &  &  & \hspace{3em}\subseteq\bigcup_{\left(\v{\xi},\eta\right)\perp\left(\v 0,-1\right)}\left\{ \left(\left(\v x,y-t\right),\left(\v{\xi},\eta\right)\right):\right.\\
 &  &  & \hspace{8em}\left.\left(\left(\v x,y\right),\left(\v{\xi},\eta\right)\right)\in WF\left\{ w\left(\v x-\v{\gamma}\left(y\right)\right)\right\} ,t\ge0\right\} \\
 &  &  & \hspace{3em}=\bigcup_{\v{\xi}\in\mathbb{R}^{n}}\left\{ \left(\left(\v x,y-t\right),\left(\v{\xi},0\right)\right):\left(\left(\v x,y\right),\left(\v{\xi},0\right)\right)\in\chi^{\star}WF\left(w\right),t\ge0\right\} \\
 &  &  & \hspace{3em}=\bigcup_{\v{\xi}\in\mathbb{R}^{n}}\left\{ \left(\left(\v x,y-t\right),\left(\v{\xi},0\right)\right):\left(\v x-\v{\gamma}\left(y\right),\v{\xi}\right)\in WF\left(w\right),\v{\xi}\perp\v{\gamma}\left(y\right),t\ge0\right\} .
\end{flalign*}
Finally,
\begin{flalign*}
 & \hfill &  & WF\left\{ \left.\int_{0}^{\infty}w\left(\v x-\v{\gamma}\left(y+t\right)\right)\, dt\right|_{y=0}\right\}  & \hfill\\
 &  &  & \hspace{2em}\subseteq\chi_{0}^{\star}WF\left\{ w\left(\v x-\v{\gamma}\left(y\right)\right)\right\} \cup\chi_{0}^{\star}\left\{ \left(\left(\v x,y-t\right),\left(\v{\xi},0\right)\right):\right.\\
 &  &  & \hspace{3em}\left.\xi\in\v{\gamma}\left(y\right)^{\perp},\left(\v x-\v{\gamma}\left(y\right),\v{\xi}\right)\in WF\left(w\right),t\ge0\right\} \\
 &  &  & \hspace{2em}=WF\left(w\right)\cup\\
 &  &  & \hspace{3em}\left\{ \left(\v x,\v{\xi}\right)\,\middle|\,\exists t\ge0:\left(\v x-\v{\gamma}\left(t\right),\v{\xi}\right)\in WF\left(w\right),\v{\xi}\in\v{\gamma}\left(t\right)^{\perp}\right\} . & \qedhere
\end{flalign*}

\end{proof}
It should be noted that the integral \ref{eq:curveintegral} reduces
to the distributional anti-partial derivative developed in the previous
section when $\v{\gamma}$ parametrizes a ray:
\[
\v{\gamma}\left(t\right)=t\v v,\quad t\ge0,\v v\in\mathcal{S}^{n-1}.
\]
Furthermore, the result obtained in \ref{thm:propagation-curve} in
this case is:
\begin{flalign*}
 & \hfill &  & WF\left\{ \int_{0}^{\infty}w\left(\v x-\v{\gamma}\left(t\right)\right)\upsilon\left(t\right)\, dt\right\} \backslash WF\left(w\right) & \hfill\\
 &  &  & \hspace{3em}\subseteq\left\{ \left(\v x,\v{\xi}\right)\,\middle|\,\exists t\ge0:\v{\xi}\perp\gamma'\left(t\right)\And\left(\v x-\v{\gamma}\left(t\right),\v{\xi}\right)\in WF\left(w\right)\right\} \\
 &  &  & \hspace{3em}=\left\{ \left(\v x,\v{\xi}\right)\,\middle|\,\exists t\ge0:\v{\xi}\perp\v v\And\left(\v x-t\v v,\v{\xi}\right)\in WF\left(w\right)\right\} ,
\end{flalign*}
which is an equivalent formulation to \ref{eq:propagation}.

\newpage{}

\bibliographystyle{plain}
\bibliography{/home/caretaker82/Mathematics/bibtexdb}

\end{document}